\documentclass[12pt,reqno]{amsart}
\usepackage{amsmath,amsthm,amsfonts,amssymb,amscd,  multicol, verbatim, enumerate, stackrel}
\usepackage{verbatim}
\usepackage{fullpage}
\input{xy}
\usepackage{xcolor}
\usepackage[english]{babel}
\usepackage{cite}
\usepackage{amsfonts}
\usepackage{latexsym}
\usepackage{amsthm}
\usepackage{amsopn}
\usepackage{amsmath}
\usepackage[all]{xy}

\usepackage{hyperref}

\usepackage{tikz}

\makeatletter
\@namedef{subjclassname@2020}{%
  \textup{2020} Mathematics Subject Classification}
\makeatother

\usepackage{verbatim}
\usepackage{amssymb}
\usepackage{mathrsfs}
\usepackage{float}
\usepackage{fullpage}
\usepackage{amscd}
\usepackage{amscd}
\usepackage{color}

\usepackage{graphicx}

\theoremstyle{plain}
\newtheorem*{thm*}{Theorem}
\newtheorem*{con*}{Conjecture}
\newtheorem{thm}{Theorem}[section]

\newtheorem{lem}[thm]{Lemma}
\newtheorem{con}[thm]{Conjecture}

\theoremstyle{definition}
\newtheorem{defn}[thm]{Definition}

\theoremstyle{remark}
\newtheorem{rmk}[thm]{Remark}

\title{ A note on Combinatorial Invariance of Kazhdan--Lusztig polynomials}
\author[Esposito, Marietti]{Francesco Esposito, Mario Marietti}

\date{}

\address{Grant T. Barkley, Department of Mathematics, Harvard University, Cambridge, MA, USA}
\email{gbarkley@math.harvard.edu}
\address{Francesco Esposito, Dipartimento di Matematica, Universit\`a degli Studi di Padova,
via Trieste 63, 35121 Padova, Italy}
\email{esposito@math.unipd.it}
\address{Christian Gaetz, Department of Mathematics, University of California, Berkeley, CA, USA}
\email{gaetz@berkeley.edu}
\address{Mario Marietti, Dipartimento  di Ingegneria Industriale e Scienze Matematiche, Universit\`a Politecnica delle Marche, via Brecce Bianche, 60131 Ancona,  Italy}
\email{m.marietti@univpm.it}


\begin{document}

\maketitle

\begin{abstract}
We introduce the concepts of an amazing hypercube decomposition and a double shortcut for it, and use these new ideas to formulate a conjecture implying the Combinatorial Invariance Conjecture of the Kazhdan--Lusztig polynomials for the symmetric group.  This conjecture has the advantage of being combinatorial in nature. The appendix by Grant T. Barkley and Christian Gaetz discusses the related notion of double hypercubes and proves an analogous conjecture for these in the case of co-elementary intervals. 
\end{abstract}

\section{Introduction}
The Combinatorial Invariance Conjecture for Kazhdan--Lusztig polynomials was formulated independently by Lusztig and Dyer in the '80's. Its version for the symmetric group $W$ is the following.
\begin{con}
\label{CIC}
Let $u,v \in W$. The Kazhdan--Lusztig polynomial $P_{u,v}(q)$ (or, equivalently, the Kazhdan--Lusztig $\widetilde{R}$-polynomial $\widetilde{R}_{u,v}(q)$) depends only on the isomorphism class of the Bruhat interval $[u,v]$ as a poset. 
\end{con}
It has been the focus of active research for the past forty years (we refer the reader to \cite{BreOPAC} for more details) and has recently received important new inputs by  \cite{BBDVW} and \cite{DVBBZTTBBJLWHK}, where the concept of a hypercube decomposition was introduced and used to give a conjectural formula for the Kazhdan--Lusztig polynomials of the symmetric group that implies Conjecture~\ref{CIC}. 

Successively, a conjectural formula for $\widetilde{R}$-polynomial $\widetilde{R}_{u,v}(q)$ of the symmetric group was given in terms of shortcuts (\cite{BM}) of a hypercube decomposition. Also this formula would imply  Conjecture~\ref{CIC}.  A further work studying hypercube decompositions is \cite{BG}, where the authors give another conjecture implying Conjecture~\ref{CIC} and prove combinatorial invariance for elementary intervals, which marks an improvement of the result in \cite{BCM1} for lower Bruhat intervals. 

A new different approach (via flipclasses) is proposed in \cite{EM}. 

In this paper, we introduce the concepts of an amazing hypercube decomposition and a double shortcut for it, and use these new ideas to formulate yet another conjecture implying the conjecture in \cite{BM} and thus ultimately Conjecture~\ref{CIC}.  This new conjecture has the advantage of being combinatorial in nature, with no mention of Kazhdan--Lusztig polynomials: a purely graph theoretic property of the Bruhat interval would imply the Combinatorial Invariance Conjecture for the symmetric group. 

The conjecture presented in this paper has been verified (in a strong form) up to $S_6$ by computer calculations.

\section{Preliminaries}
In this section we recall the basic definitions and set the notation to be used. We  refer the reader to    \cite{BB} for undefined  terminology concerning Coxeter groups.

We recall the definition of a hypercube decomposition from  \cite{BBDVW}. Let $(W,S)$ be a Coxeter group of type $A_n$, i.e., the symmetric group $S_{n+1}$. We denote by $B(W)$ the Bruhat graph of $W$.

Given a set $E$, let ${\mathcal P}(E)$ denote the directed Boolean algebra on $E$, i.e. the directed graph having the power set of $E$ as vertex set and where $I\rightarrow J$ if $I$ is obtained from $J$ by removing one element.  

Let $p\in W$ and $E$ be a set of arrows of $B(W)$ all having  $p$ as target. Then  $E$ \emph{spans a hypercube} if there exists a unique embedding of directed graphs $\theta: {\mathcal P}(E) \rightarrow W$ sending the directed edge $E\setminus {\{\alpha\}} \rightarrow E$ to $\alpha$, for all $\alpha \in E$. 
Furthermore, $E$ \emph{spans a hypercube cluster} if every subset $E'$ of $E$ consisting of edges with pairwise incomparable sources (with respect to Bruhat order) spans a hypercube.

Let $u,v \in W$ and $z\in[u,v]$. We say that $z$ is an \emph{upper hypercube decomposition} of $[u,v]$ provided that:
\begin{enumerate}
\item 
$[z,v]$ is \emph{diamond complete} (with respect to $[u,v]$), meaning that, if there exist $x\in [u,v]$ and $a_1,a_2,y \in [z,v]$, $a_1 \neq a_2$, such that $x\rightarrow a_1\rightarrow y$, $x \rightarrow a_2 \rightarrow y$, then $x\in[z,v]$; 
\item   
for all $p\in [z,v]$, the set $E^p= \{x \rightarrow p : x\notin [z,v]\}$  spans a hypercube cluster.
\end{enumerate}

Similarly, taking the dual versions of the above definitions, we obtain the concept of a \emph{lower  hypercube decomposition}. 

We introduce  the $\widetilde{R}$-polynomials by the following combinatorial interpretation due to Dyer (see \cite{DyeComp} and also \cite[Theorem~5.3.4]{BB}). Given a path $\Gamma= (x_0 \rightarrow x_1 \rightarrow \cdots \rightarrow x_r)$, we denote by $|\Gamma|$ the length of $\Gamma$ (i.e., $|\Gamma|= r$), and we define  $\{x_i : i\in[0,r]\}$ to be the {\em support} of $\Gamma$, denoted $\operatorname{supp} (\Gamma)$. 

\begin{thm}
\label{Dyertilde}
Let  $\preceq $ be a reflection ordering, and  $u,v \in \mathfrak S_n$. Then
$$\widetilde{R}_{u,v}(q)=\sum q^{|\Gamma|},$$
where the sum is over all increasing paths $\Gamma$ from $u$ to $v$.
\end{thm}

Following \cite{BM}, we give the next definition. We denote by $\operatorname{d}(x,y)$ the distance function of $B(W)$, i.e., the minimum of the lengths of paths  from $x$ to $y$. 
\begin{defn}
\label{scorcia}
 Let  $u,v,z \in W$ with  $z \in [u,v]$. We let 
$$
W_{[u,v]}^z := 
\{p\in[z,v] : \operatorname{supp} (\Gamma) \cap [z,v]= \{p\} \textrm{ for all paths $\Gamma$ from $u$ to $p$ with $|\Gamma|=\operatorname{d}(u,p)$}\},
$$
 and 
$$
\widetilde{R}_{u,v}^z(q) :=
\sum_{p\in W_{[u,v]}^z}  q^{\operatorname{d}(u,p)} \widetilde{R}_{p,v}(q).
$$
We call the elements in $W_{[u,v]}^z$ the \emph{(upper) shortcuts of $[u,v]$ with respect to $z$}. Furthermore, we say that an element $z$ in $[u,v]$ is an \emph{(upper) $R$-element for $[u,v]$} if $\widetilde{R}_{u,v}^z=\widetilde{R}_{u,v}$. 
 \end{defn}

\begin{rmk}
\label{decomposizioni canoniche}
The elements $\min([u,v]\cap W_{S\setminus \{s_{n-1}\}}v) $, $\min([u,v]\cap W_{S\setminus \{s_{1}\}} v)$, $\min([u,v]\cap  v W_{S\setminus \{s_{n-1}\}} )$, and $\min([u,v]\cap v  W_{S\setminus \{s_{1}\}} v)$ are always upper hypercube decompositions for the interval $[u,v]$ (possibly  not mutually distinct). We call these  elements  the {\em standard} upper hypercube decompositions for  $[u,v]$. Note that these hypercube decompositions are called canonical in \cite{BM}. 
\end{rmk}

\begin{rmk}
We note that the concept of an upper shortcut can be equivalently defined in the following way. Let $u,v \in W$, and $z$ be  an upper hypercube decomposition of $[u,v]$. An element $p$ in $[z,v]$ is an upper shortcut of $[u,v]$ with respect to $z$ if $\operatorname{d}(u,p) < \operatorname{d}(u,x)$ for all $x$ in $[z,p]$ satisfying $\operatorname{d}(x,p) = 1$.
\end{rmk}

\section{Amazing hypercube decompositions and double shortcuts}

The following definition strengthens the definition of a \emph{join hypercube decomposition} appearing in \cite{BM} (indeed, it is a stable version of it).

\begin{defn}
Let $u,v \in W$. We say that an upper hypercube decomposition $z$ of $[u,v]$ is  \emph{amazing}  provided that,
for all $x \in [u,v]$, the intersection $[z,v]\cap [x,v]$ has a minimum $z\vee x$, and this minimum $z \vee x$ is a (necessarily  amazing)   hypercube decomposition of $[x,v]$. Furthermore, we say that an amazing hypercube decomposition $z$ of $[u,v]$ is an \emph{(upper) amazing $R$-element for $[u,v]$} if $z \vee x$ is an (upper) $R$-element for $[x,v]$ for all $x \in [u,v]$.
\end{defn}
\begin{rmk}
Let $u,v \in W$ and $z$ be a standard hypercube decomposition of $[u,v]$. Since $z \vee x$ exists and is a standard hypercube decomposition of $[x,v]$ (of the same kind), we have that $z$ is an amazing hypercube decomposition, and, by \cite[Corollary~3.10]{BM}, an amazing $R$-element. 
\end{rmk}
The following conjecture is a weak version of  Conjecture 6.1  of \cite{BM}, which has been verified up to $W=S_6$.
\begin{con}
\label{congettura}
 Let $W$ be a Coxeter group of type $A$, $u,v \in W$, and $z\in[u,v]$ be an amazing upper hypercube decomposition. Then $z$ is an $R$-element, i.e.
$$
 \widetilde{R}_{u,v}=\sum_{p\in W_{[u,v]}^z}  q^{\operatorname{d}(u,p)} \widetilde{R}_{p,v}.
$$
 \end{con} 

Conjecture~\ref{congettura} implies the Combinatorial Invariance Conjecture.
\begin{defn}
Let $u,v \in W$, and $z$ and $z'$ be two amazing upper hypercube decomposition of $[u,v]$. We say that an element $b$ in $[z,v]\cap [z',v]$ is an (upper)  $(z,z')$-double shortcut of $[u,v]$ if there exists $p\in W^z_{[u,v]}$ such that $b\in W^{z'\vee p}_{[p,v]}$.  We denote by $DS(z,z')$ the multiset 
$$\{(\operatorname{d}(u,p) + \operatorname{d}(p,b) , b ) : p\in W^z_{[u,v]} \text{ and } b\in W^{z'\vee p}_{[p,v]}  \}.$$
\end{defn}

Figure \ref{espofinalmente} depicts an interval $[u,v]$, two amazing upper hypercube decompositions  $z$ and $z'$  of $[u,v]$, an upper shortcut $p$ of $[u,v]$ with respect to $z$, and an upper shortcut $b$ of $[p,v]$ with respect to $z'\vee p$. The element  $b$  is hence a $(z,z')$-double shortcut of $[u,v]$.
\begin{figure}[h]
    \centering
\scalebox{.9}{
\begin{tikzpicture}

\draw (0,0) node {$\bullet$};
\draw (0,-0.2) node {$u$};

\draw (0,8) node {$\bullet$};
\draw (0,8+0.2) node {$v$};

\draw (4,2) node {$\bullet$};
\draw (4,2-0.2) node {$z'$};

\draw (-3,2) node {$\bullet$};
\draw (-3,2-0.2) node {$z$};

\draw (-2.5,3.5) node {$\bullet$};
\draw (-2.5,3.5-0.3) node {$p$};


\draw (-0.5,4.5) node {$\bullet$};
\draw (-0.5+0.7,4.5) node {$z'\vee p$};

\draw (0,7) node {$\bullet$};
\draw (0,7+0.3) node {$b$};

\draw (5,4) to[out=90,in=-15] node [sloped,right] {} (0,8);
\draw (0,0) to[out=0,in=-90] node [sloped,right] {} (5,4);

\draw (-5,4) to[out=90,in=180+15] node [sloped,right] {} (0,8);
\draw (0,0) to[out=180,in=-90] node [sloped,right] {} (-5,4);

\draw (-3,2) to[out=180-25,in=180+15] node [sloped,right] {} (0,8);
\draw (-3,2) to[out=0,in=180+45] node [sloped,right] {} (0.5,3.2);
\draw (0.5,3.2) to[out=45,in=-15] node [sloped,right] {} (0,8);

\draw (4,2) to[out=180,in=-45] node [sloped,left] {} (-0.5,4.5);
\draw (-0.5,4.5) to[out=90+45,in=180+15] node [sloped,right] {} (0,8);
\draw (4,2) to[out=40,in=-15] node [sloped,right] {} (0,8);

\draw (-2.5,3.5) to[out=180-25,in=180+15] node [sloped,right] {} (0,8);
\draw (-2.5,3.5) to[out=0,in=-115] node [sloped,right] {} (-0.5,4.5);
\draw (-0.5,4.5) to[out=65,in=-15] node [sloped,right] {} (0,8);

\end{tikzpicture}
 }
    \caption{A $(z,z')$-double shortcut of $[u,v]$.}
    \label{espofinalmente}
\end{figure}

\begin{thm}
\label{bologna}
 Let $W$ be a Coxeter group of type $A$, $u,v \in W$, and $z,z'\in[u,v]$ be two amazing upper hypercube decomposition. Suppose that 
\begin{enumerate}
\item
\label{uno}
 $z$ is an amazing $R$-element,
\item $z'\vee x$ is  an (amazing) $R$-element  for all $x\in [u,v]\setminus\{u\}$,
\label{due}
\item 
\label{tre}
$DS(z,z') =DS(z',z)$ (as multisets).
\end{enumerate}
Then, $z'$ is an (amazing) $R$-element.
\end{thm}
\begin{proof}
We have:
\begin{eqnarray*}
 \widetilde{R}_{u,v}&=& \sum_{p\in W_{[u,v]}^{z}}  q^{\operatorname{d}(u,p)} \widetilde{R}_{p,v}\\
& =&  \sum_{p\in W_{[u,v]}^{z}}  q^{\operatorname{d}(u,p)} \big( \sum_{b\in W_{[p,v]}^{z'\vee p}}  q^{\operatorname{d}(p,b)} \widetilde{R}_{b,v} \big) \\
&=&  \sum_{(a,b) \in DS(z,z')}  q^a \widetilde{R}_{b,v}\\
&=&  \sum_{(a,b) \in DS(z',z)}  q^a \widetilde{R}_{b,v}\\
& =&  \sum_{p'\in W_{[u,v]}^{z'}}  q^{\operatorname{d}(u,p')} \big( \sum_{b\in W_{[p',v]}^{z\vee p'}}  q^{\operatorname{d}(p',b)} \widetilde{R}_{b,v} \big) \\
&=& \sum_{p'\in W_{[u,v]}^{z'}}  q^{\operatorname{d}(u,p')} \widetilde{R}_{p',v},
\end{eqnarray*}
where the first equation follows by (\ref{uno}),  the second equation follows by (\ref{due}), the third equation by the definition of the multiset $DS(z,z')$, the fourth by (\ref{tre}), the fifth by the definition of the multiset $DS(z',z)$, and the sixth equation  follows by (\ref{uno}).
\end{proof}

Fix $u,v \in W$. Let us define an equivalence relation on the set of amazing hypercube decompositions of $[u,v]$   as the transitive closure of the relation $z \sim z'$ if $DS(z,z')= DS(z',z)$.   
\begin{con}
\label{congetturaEM0}
 Let $W$ be a Coxeter group of type $A$. 
The  above relation is trivial (i.e., it has only one equivalence class).
 \end{con} 
The following conjecture is a weakening of Conjecture~\ref{congetturaEM0}.
\begin{con}
\label{congetturaEM}
 Let $W$ be a Coxeter group of type $A$. 
Every equivalence class of the above relation contains an amazing $R$-element.
 \end{con}

\begin{rmk}
Up to $S_6$, we checked by computer that, in fact,  a much stronger statement than Conjecture~\ref{congetturaEM} holds:
$DS(z,z')=DS(z',z)$ for all amazing hypercube decompositions of any interval.
\end{rmk}
 \begin{thm}
 \label{thm:conj-implies-conj}
 Conjecture~\ref{congetturaEM} implies Conjecture~\ref{congettura}.
 \end{thm}
\begin{proof}
Towards a contradiction, let $[u,v]$ be an interval of minimal length having an amazing hypercube decomposition $z$ that fails to be an $R$-element. By hypothesis, there exists a sequence $(z_0, \ldots, z_r)$ of amazing hypercube decompositions such that $z_0 = z$,  and $z_r$ is an amazing $R$-element, and $DS(z_{i-1},z_{i})= DS(z_{i},z_{i-1})$ for each $i\in [r]$.  One gets a contradiction by iteratively applying Theorem~\ref{bologna}. 
\end{proof}

The following result supports Conjecture~\ref{congetturaEM}.
\begin{thm}
Let $(W, S)$, $(W_1, S_1)$, and $(W_2, S_2)$ be three Coxeter systems of type $A$ and $u, v \in W$, $u_i, v_i \in W_i$ for $i = 1, 2$ be such that $[u, v] \cong [u_1, v_1] \times [u_2, v_2]$ (isomorphic as
posets), and let $x\mapsto (x_1, x_2)$ be an isomorphism. Let $x, y \in [u, v]$ and $z,z'$ be two amazing hypercube decompositions of $[u,v]$ which correspond to $ (z_1, z_2)$ and $ (z'_1, z'_2)$, respectively (see \cite[Theorem~5.6]{BM}).  Suppose $DS(z_1,z_1')=DS(z_1',z_1)$ and $DS(z_2,z_2')=DS(z_2',z_2)$. Then $$DS(z,z')=DS(z',z).$$
\end{thm}
\begin{proof}
By \cite[Proposition~5.7]{BM}), for $p,p',b,b' \in [u,v]$, we have:
\begin{itemize}
\item  $p\in W^z_{[u,v]}$ if and only if  $p_1\in W^{z_1}_{[u_1,v_1]}$ and $p_2\in W^{z_2}_{[u_2,v_2]}$; 
\item  $p'\in W^{z'}_{[u,v]}$ if and only if  $p'_1\in W^{z'_1}_{[u_1,v_1]}$ and $p'_2\in W^{z'_2}_{[u_2,v_2]}$;
\item  $b\in W^{z'\vee p}_{[p,v]}$ if and only if  $b_1\in W^{z'_1 \vee p_1}_{[p_1,v_1]}$ and $b_2\in W^{z'_2 \vee p_2}_{[p_2,v_2]}$; 
\item  $b'\in W^{z'\vee p'}_{[p',v]}$ if and only if  $b'_1\in W^{z_1 \vee p'_1}_{[p'_1,v_1]}$ and $b'_2\in W^{z_2 \vee p'_2}_{[p'_2,v_2]}$.
\end{itemize}
Hence, $DS(z_1,z_1')=DS(z_1',z_1)$ and $DS(z_2,z_2')=DS(z_2',z_2)$ imply the assertion.
\end{proof}

\bigskip
{\bf Acknowledgments:} 
We would like to thank Grant T. Barkley, Francesco Brenti, Matthew Dyer, and Christian Gaetz for stimulating discussions we had during the Workshop \lq\lq Bruhat order: recent developments and open problems\rq\rq, held at the University of  Bologna from 15 April 2024 to 19 April 2024. The authors are members of the I.N.D.A.M. group GNSAGA.
 
\section*{Appendix: double hypercubes \\ By Grant T. Barkley and Christian Gaetz}
\newcommand{\cH}{\mathcal{H}}
\newcommand{\sH}{\mathscr{H}}

In this appendix, we will prove an analog of Conjecture~\ref{congetturaEM0} for \emph{co-elementary intervals}. We say $[u,v] \subset S_n$ is \emph{co-simple} if the set of roots corresponding to the lower cover relations of $v$ in $[u,v]$ is linearly independent and \emph{co-elementary} if $[u,v]$ is isomorphic as a poset to $[u',v']$, where $[u',v']$ is a co-simple interval in some symmetric group $S_m$ (see \cite{BG}, where the poset dual notions of \emph{simple} and \emph{elementary} intervals are used). Every interval of the form $[u,w_0]$ is co-simple. 

In this appendix, rather than directly examining double shortcuts, we will study an analog that we call \emph{double hypercubes}. Given an upper hypercube decomposition $z\in [u,v]$, we let
\begin{align*}
\sH^z_{[u,v]} &:= \{ (\cH, p) : p\in [z,v] \text{ and } \cH \text{ is a hypercube in $[u,v]$ containing $u$,} \\& \text{spanned by an antichain, and with }\cH\cap [z,v] = \{p\}    \}.    
\end{align*}
We conjecture that for amazing hypercube decompositions, the map sending $(\cH,p)\in \sH^z_{[u,v]}$ to $p$ is a bijection from $\sH^z_{[u,v]}$ to $W^z_{[u,v]}$. We will state our result in terms of $\sH^z_{[u,v]}$; together with the conjectured bijection, this would also imply cases of Conjecture~\ref{congetturaEM0}. 

\begin{defn}
    Let $z$ and $z'$ be upper amazing hypercube decompositions of $[u,v]$. 
    We denote by $DH(z,z')$ the multiset
    \[ \{ (|\cH_1| + |\cH_2|, b) : (\cH_1,p) \in \sH_{[u,v]}^z \text{ and } (\cH_2,b)\in \sH^{z'\vee p}_{[p,v]} \}, \]
    where $|\mathcal{H}|$ denotes the rank of the hypercube $\mathcal{H}$. An element of $DH(z,z')$ is called a \emph{double hypercube}.
\end{defn}

We now present Theorem~\ref{thm:coelementary-double-hc}, the main result of this appendix. This proves a direct analog of Conjecture~\ref{congetturaEM0}, and gives a new proof of combinatorial invariance of $R$- and $\tilde{R}$-polynomials for co-elementary intervals via the analog of Theorem~\ref{thm:conj-implies-conj}.

\begin{thm}
\label{thm:coelementary-double-hc}
    Let $[u,v]$ be a co-elementary interval. Let $z$ and $z'$ be upper amazing hypercube decompositions which are strong in the sense of \cite[Def. 3.4]{BG}. Then $DH(z,z')=DH(z',z)$ as multisets.
\end{thm}
\begin{proof}

We first note that the claim reduces immediately to the case of co-simple intervals, since everything in sight is preserved by poset isomorphism. 

\newcommand{\cC}{\mathcal{C}}
\newcommand{\tR}{\widetilde{R}}
We refer to \cite[Chapter 5]{BB} or \cite{BG} for the notions of reflections labeling Bruhat graph edges, reflection orders, and increasing paths in the Bruhat graph with respect to a reflection order. We write $\cC_{[x,y]}$ for the set of reflections labeling the coatoms of an interval $[x,y]$. We first prove the following lemma.

\begin{lem}\label{lem:incpaths}
    Let $[u,v]$ be co-simple and $z\in [u,v]$  be so that $[z,v]$ is diamond complete in $[u,v]$. Let $\leq_1$ and $\leq_2$ be reflection orders 
    so that if $t \in \cC_{[u,v]}\setminus \cC_{[z,v]}$ and $t'\in \cC_{[z,v]}$, then $t\leq_1 t'$ and $t\leq_2 t'$.
    Then for any $p\in [z,v]$ and $k\in \mathbb{N}$, the number of $\leq_i$-increasing length-$k$ paths $\Gamma$ from $u$ to $p$ such that $\operatorname{supp}(\Gamma)\cap [z,v] = \{p\}$ is independent of $i\in \{1,2\}$.
\end{lem}
\begin{proof}[Proof of Lemma~\ref{lem:incpaths}]
    For $p\in [z,v]$ write $a_i^{p,k}$ for the number of length-$k$ paths $\Gamma$ from $u$ to $p$ such that $\operatorname{supp}(\Gamma)\cap [z,v] = \{p\}$ which are increasing with respect to $\leq_i$. We induct on the size $|[u,p]|$. If $p=z$, then $a_i^{p,k}$ is the coefficient of $q^k$ in $\tR_{u,p}$ (see Theorem~\ref{Dyertilde}), which is independent of $i$. For arbitrary $p\in [z,v]$, we have the identity
    \[ \sum_{\substack{x\in [z,p] \\ k \in \mathbb{N} }} a_i^{x,k}q^k \tR_{x,p} = \tR_{u,p}. \tag{$\ast$}
    \label{eqn:aident} \]
    The left-hand side counts pairs consisting of  $\leq_i$-increasing paths from $u$ to $x$ which intersect $[z,p]$ at $\{x\}$ and $\leq_i$-increasing paths from $x$ to $p$. The right-hand side counts $\leq_i$ increasing paths from $u$ to $p$. The identity (\ref{eqn:aident}) then follows by the hypothesis on $\leq_i$, which implies that the concatenation of the two paths in a pair counted by the left-hand side is an increasing path (cf. the proof of \cite[Corollary 4.6]{BG}). 

    Now inductively we have that
    \[ \sum_{\substack{x\in [z,p] \\ x\neq p \\ k \in \mathbb{N} }} a_1^{x,k}q^k \tR_{x,p} = \sum_{\substack{x\in [z,p] \\ x\neq p \\ k \in \mathbb{N} }} a_2^{x,k}q^k \tR_{x,p}. \]
    Evidently (\ref{eqn:aident}) implies
    \[ \sum_{\substack{x\in [z,p] \\ k \in \mathbb{N} }} a_1^{x,k}q^k \tR_{x,p} = \sum_{\substack{x\in [z,p] \\ k \in \mathbb{N} }} a_2^{x,k}q^k \tR_{x,p}. \]
    Subtracting the two equations gives
    \[ \sum_{k\in \mathbb{N}} a_1^{p,k}q^k = \sum_{k\in \mathbb{N}} a_2^{p,k}q^k, \]
    from which the claim follows.

 \end{proof}

We now pick two amazing hypercube decompositions $z,z'\in [u,v]$. The remainder of the proof is simply an application of results from \cite{BG} to construct two reflection orders so that their increasing paths correspond with elements of $DH(z,z')$ and $DH(z',z)$, respectively. We then apply Lemma~\ref{lem:incpaths} to deduce the equality $DH(z,z')=DH(z',z)$. Note that \cite{BG} uses lower hypercube decompositions rather than upper, so the statements there are dual to the ones being used below.

By \cite[Lemma 4.5]{BG} and the proof of \cite[Corollary 4.6]{BG}, there is a reflection order $\leq_{z}$ with the property that any edge in $[z,v]\setminus[z',v]$ precedes any edge in $[z,v]\cap [z',v]$, and any edge in $[u,v]\setminus [z,v]$ precedes any edge in $[z,v]$. There is also an analogous reflection order $\leq_{z'}$. Hence $\leq_z$ satisfies property (E) from \cite{BG} both for the hypercube decomposition $z \in [u,v]$ and for the hypercube decomposition $z\vee z' \in [z,v]$. We deduce from two applications of \cite[Theorem 3.10]{BG} that the $\leq_z$-increasing paths $\Gamma$ from $u$ to $v$ are in bijection with tuples $(\cH_1,p,\cH_2,b,\gamma)$, where $(\cH_1,p)$ is in $\mathscr{H}_{[u,v]}^z$ and $(\cH_2,b)$ is in $\mathscr{H}_{[p,v]}^{z\vee z'}$ and $\gamma$ is a $\leq_z$-increasing path from $b$ to $v$. Furthermore the length of $\Gamma$ is $|\cH_1|+|\cH_2|+|\gamma|$, and from the proof of \cite[Theorem 3.10]{BG}, the path $\gamma$ is a subpath of $\Gamma$. Hence we deduce that the multiset
\[ \{(|\Gamma|,b) : \text{$\Gamma$ is $\leq_z$-increasing from $u$ to $b$},~\mathrm{supp}(\Gamma)\cap [z\vee z',v] = \{b\} \} \]
coincides with the multiset $DH(z,z')$. Similarly, the multiset
\[ \{(|\Gamma|,b) : \text{$\Gamma$ is $\leq_{z'}$-increasing from $u$ to $b$},~\mathrm{supp}(\Gamma)\cap [z\vee z',v] = \{b\} \} \]
coincides with the multiset $DH(z',z)$. By Lemma \ref{lem:incpaths}, we deduce $DH(z,z')=DH(z',z)$.
    
\end{proof}

\end{document}